\documentclass[12pt]{article}
\usepackage{amsmath, amsthm, amsfonts, amssymb}
\usepackage{mathrsfs}
\usepackage{bbm}
\usepackage{amssymb}
\usepackage{txfonts}

\newlength{\defbaselineskip}
\newcounter{marnote}

\newcommand{\setlinespacing}[1]%
           {\setlength{\baselineskip}{#1 \defbaselineskip}}

\theoremstyle{plain}
\newtheorem{theorem}{Theorem}[section]

\newtheorem{lemma}[theorem]{Lemma}
\newtheorem{prop}[theorem]{Proposition}
\theoremstyle{definition}
\newtheorem{definition}{Definition}[section]
\theoremstyle{remark}
\newtheorem{remark}{Remark}[section]
\numberwithin{equation}{section}

\begin{document}

\title{On the exterior Dirichlet problem for Hessian equations}

\author{Jiguang Bao\footnote{School of Mathematical Sciences, Beijing Normal University, Laboratory of Mathematics and Complex Systems, Ministry of
   Education, Beijing 100875,
China. Email: jgbao@bnu.edu.cn; hgli@bnu.edu.cn.}\quad Haigang
Li\footnotemark[1]\quad and\quad Yanyan Li\footnote{Department of
Mathematics, Rutgers University, 110 Frelinghuysen Rd, Piscataway,
NJ 08854, USA. Email: yyli@math.rutgers.edu.} }

\date{}

\maketitle

\begin{abstract}
In this paper, we establish a theorem on the existence of the
solutions of the exterior Dirichlet problem for Hessian equations
with prescribed asymptotic behavior at infinity. This extends a
result of Caffarelli and Li in \cite{cl} for the Monge-Amp\`{e}re
equation to Hessian equations.

\end{abstract}

\section{Introduction}

In this paper, we consider the solvability of the Dirichlet problem
for Hessian equations
\begin{equation}\label{sigmak}
\sigma_{k}(\lambda(D^{2}u))=1
\end{equation}
on exterior domains $\mathbb{R}^{n}\setminus{D}$, where $D$ is a
bounded open set in $\mathbb{R}^{n}$, $n\geq3$, $\lambda(D^{2}u)$
denotes the eigenvalues $\lambda_{1},\cdots,\lambda_{n}$ of the
Hessian matrix of $u$. Here
$$\sigma_{k}(\lambda)=\sum_{1\leq{i}_{1}<\cdots<i_{k}\leq{n}}\lambda_{i_{1}}\cdots\lambda_{i_{k}}$$
is the $k$-th elementary symmetric function of $n$ variations,
$k=1,\cdots,n$. Note that the case $k=1$ corresponds to the
Poisson's equation, which is a linear equation. There have been
extensive literatures on the exterior Dirichlet problem for linear
elliptic equations of second order, see \cite{ms} and the references
therein. For $2\leq{k}\leq{n}$, the Hessian equation \eqref{sigmak}
is an important class of fully nonlinear elliptic equations.
Especially, for $k=n$, we have the Monge-Amp\`{e}re equation
$\det(D^2u)=1$.

For the Monge-Amp\`{e}re equation, a classical theorem of
J\"{o}rgens (\cite{j}), Calabi (\cite{ce}), and Pogorelov (\cite{p})
states that any classical convex solution of $\det(D^2u)=1$ in
$\mathbb{R}^n$ must be a quadratic polynomial. A simpler and more
analytic proof, along the lines of affine geometry, was later given
by Cheng and Yau \cite{cy}. Caffarelli  \cite{c} extended the result
for classical solutions to viscosity solutions. Another proof of
this theorem was given by Jost and Yin in \cite{jx}. Trudinger and
Wang \cite{tw} proved that if $\Omega$ is an open convex subset of
$\mathbb{R}^{n}$ and $u$ is a convex $C^{2}$ solution of
$\det(D^{2}u)=1$ in $\Omega$ with
$\lim_{x\rightarrow\partial\Omega}u(x)=\infty$, then
$\Omega=\mathbb{R}^{n}$ and $u$ is quadratic.

Caffarelli and the third author \cite{cl} extended the
J\"{o}rgens-Calabi-Pogorelov theorem to exterior domains. They
proved that if $u$ is a convex viscosity solution of
$\det(D^{2}u)=1$ outside a bounded subset of $\mathbb{R}^{n}$,
$n\geq3$, then there exist a $n\times{n}$ real symmetric positive
definite matrix $A$, a vector $b\in\mathbb{R}^{n}$, and a constant
$c\in\mathbb{R}$ such that
\begin{equation}\label{asymptotic_cl}
\limsup_{|x|\rightarrow\infty}~
\left(|x|^{n-2}\left|~u(x)-\left(\frac{1}{2}x^{T}Ax+b\cdot{x}+c\right)~\right|\right)<\infty.
\end{equation}
With this prescribed asymptotic behavior at infinity, an existence
result for the exterior Dirichlet problem for the Monge-Amp\`{e}re
equation in $\mathbb{R}^{n}$, $n\geq3$, was also established in
\cite{cl}. In this paper, we will extend the existence theorem to
the Dirichlet problem for Hessian equations \eqref{sigmak} with
$2\leq{k}\leq{n}-1$ on exterior domains, with an appropriate
asymptotic behavior at infinity. In dimension two, similar problems
were studied by Ferrer, Mart\'{i}nez and Mil\'{a}n in
\cite{fmm1,fmm2} using complex variable method. See also Delano\"{e}
\cite{d}.

We remark that for the case that $A=c^{*}I$, where
$$ c^{*}=(C_{n}^{k})^{-1/k}, \ \ \ \ C_{n}^{k}=\frac{n!}{(n-k)!k!}, $$
$I$ is the $n\times{n}$ identity  matrix and $1\leq{k}\leq{n}$, the
exterior Dirichlet problem of Hessian equation \eqref{sigmak} has
been investigated in \cite{d1,db}. For interior domains, there have
been many well known results on the solvability of Hessian
equations. For instance, Caffarelli, Nirenberg and Spruck \cite{cns}
established the classical solvability of the Dirichlet problem,
Trudinger \cite{t2} proved the existence and uniqueness of weak
solutions, and Urbas \cite{u} demonstrated the existence of
viscosity solutions. Jian \cite{jian} studied the Hessian equations
with infinite Dirichlet boundary value conditions.

For readers' convenience, we recall the definition of viscosity
solutions to Hessian equations (see \cite{cc,u} and the references
therein). We say that a function
$u\in{C}^{2}(\mathbb{R}^{n}\setminus{\overline{D}})$ is admissible
(or $k$-convex) if $\lambda(D^{2}u)\in\overline{\Gamma}_{k}$ in
$\mathbb{R}^{n}\setminus{\overline{D}}$, where $\Gamma_{k}$ is the
connected component of
$\{\lambda\in\mathbb{R}^{n}~|~\sigma_{k}(\lambda)>0\}$ containing
$$\Gamma^{+}=\{\lambda\in\mathbb{R}^{n}~|~\lambda_{i}>0,i=1,\cdots,n\}.$$
It is well known that $\Gamma_{k}$ is a convex symmetric cone with
vertex at the origin. Moreover,
$$\Gamma_{k}=\{\lambda\in\mathbb{R}^{n}~|~\sigma_{j}(\lambda)>0,\ \mbox{for\ \ all}~j=1,\cdots,k\}.$$
See \cite{cns,t1}. Clearly, $\Gamma_{k}\subseteq\Gamma_{j}$ for
$k\geq{j}$, and $\Gamma_{1}$ is the half space
$\{\lambda\in\mathbb{R}^{n}~|~\lambda_{1}+\cdots+\lambda_{n}>0\}$,
while $\Gamma_{n}=\Gamma^{+}$. We use the following definitions,
which can be found in \cite{t0}.

Let $\Omega\subset\mathbb{R}^{n}$, we use $\mathrm{USC}(\Omega)$ and
$\mathrm{LSC}(\Omega)$ to denote respectively the set of upper and
lower semicontinuous real valued functions on $\Omega$.

\begin{definition}
A function $u\in\mathrm{USC}(\mathbb{R}^{n}\setminus\overline{D})$
is said to be a viscosity subsolution of equation \eqref{sigmak} in
$\mathbb{R}^{n}\setminus\overline{D}$ (or say that $u$ satisfies
$\sigma_{k}(\lambda(D^{2}u))\geq1$ in
$\mathbb{R}^{n}\setminus\overline{D}$ in the viscosity sense), if
for any function
$\psi\in{C}^{2}(\mathbb{R}^{n}\setminus\overline{D})$ and point
$\bar{x}\in\mathbb{R}^{n}\setminus\overline{D}$ satisfying
$$\psi(\bar{x})=u(\bar{x})\quad\mbox{and}\quad\psi\geq{u}~\mbox{on}~\mathbb{R}^{n}\setminus\overline{D},$$
we have
$$\sigma_{k}(\lambda(D^{2}\psi(\bar{x})))\geq1.$$
A function $u\in\mathrm{LSC}(\mathbb{R}^{n}\setminus\overline{D})$
is said to be a viscosity supersolution of \eqref{sigmak} in
$\mathbb{R}^{n}\setminus\overline{D}$ (or say that $u$ satisfies
$\sigma_{k}(\lambda(D^{2}u))\leq1$ in
$\mathbb{R}^{n}\setminus\overline{D}$ in the viscosity sense), if
for any $k$-convex function
$\psi\in{C}^{2}(\mathbb{R}^{n}\setminus\overline{D})$ and point
$\bar{x}\in\mathbb{R}^{n}\setminus\overline{D}$ satisfying
$$\psi(\bar{x})=u(\bar{x})\quad\mbox{and}\quad\psi\leq{u}~\mbox{on}~\mathbb{R}^{n}\setminus\overline{D},$$
we have
$$\sigma_{k}(\lambda(D^{2}\psi(\bar{x})))\leq1.$$
A function $u\in{C}^{0}(\mathbb{R}^{n}\setminus\overline{D})$ is
said to be a viscosity solution of \eqref{sigmak}, if it is both a
viscosity subsolution and supersolution of \eqref{sigmak}.
\end{definition}

It is well known that a function
$u\in{C}^{2}(\mathbb{R}^{n}\setminus\overline{D})$ is a viscosity
solution (respectively, subsolution, supersolution) of
\eqref{sigmak} if and only if it is a $k$-convex classical solution
(respectively, subsolution, supersolution).

\begin{definition}
Let $\varphi\in{C}^{0}(\partial{D})$. A function
$u\in\mathrm{USC}(\mathbb{R}^{n}\setminus{D})$
($u\in\mathrm{LSC}(\mathbb{R}^{n}\setminus{D})$) is said to be a
viscosity subsolution (supersolution) of the Dirichlet problem
\begin{equation}\label{dirichlet}
\begin{cases}
\sigma_{k}(\lambda(D^{2}u))=1,&\mbox{in}~\mathbb{R}^{n}\setminus{\overline{D}},\\
u=\varphi,&\mbox{on}~\partial{D},
\end{cases}
\end{equation}
if $u$ is a viscosity subsolution (supersolution) of \eqref{sigmak}
in $\mathbb{R}^{n}\setminus{\overline{D}}$ and $u\leq$ ($\geq$)
$\varphi$ on $\partial{D}$. A function
$u\in{C}^{0}(\mathbb{R}^{n}\setminus{D})$ is said to be a
viscosity solution of \eqref{dirichlet} if it is both a subsolution
and a supersolution.
\end{definition}

Let
$$\mathcal{A}_{k}=\left\{A ~\big|~ A\ \mbox{is a real}\ n\times{n}\ \mbox{symmetric positive definite matrix,}
 \mbox{ with}\ \sigma_{k}(\lambda(A))=1\right\}.$$
Our main result is
\begin{theorem}\label{thm1}
Let $D$ be a smooth, bounded, strictly convex open subset of
$\mathbb{R}^{n}$, $n\geq3$, and let
$\varphi\in{C}^{2}(\partial{D})$. Then for any given
$b\in\mathbb{R}^{n}$ and any given $A\in\mathcal{A}_{k}$ with
$2\leq{k}\leq{n}$, there exists some constant $c_{*}$, depending
only on $n,b,A, D$ and $\|\varphi\|_{C^{2}(\partial{D})}$, such
that for every $c>c_{*}$ there exists a unique viscosity solution
$u\in{C}^{0}(\mathbb{R}^{n}\setminus{D})$ of \eqref{dirichlet} and
\begin{equation}\label{asymptotic}
\limsup_{|x|\rightarrow\infty}\left(~|x|^{\theta(n-2)}\left|~u(x)-\left(\frac{1}{2}x^{T}Ax+b\cdot{x}+c\right)\right|~\right)<\infty,
\end{equation}
where $\theta\in\left[\frac{k-2}{n-2},1\right]$ is a constant
depending only on $n,k$, and $A$.
\end{theorem}

\begin{remark}
For the two cases $(\mathrm{i})$ $k=n$, the Monge-Amp\`{e}re
equations with any $A\in\mathcal{A}_{n}$; and $(\mathrm{ii})$
$2\leq{k}\leq{n}-1$, \eqref{asymptotic} with
$A=c^{*}I\in\mathcal{A}_{k}$, Theorem \ref{thm1} has been proved by
Caffarelli-Li \cite{cl} and Dai-Bao \cite{db}, respectively, where
$\theta=1$. Moreover, for the symmetric case $A=c^{*}I$, Wang-Bao
\cite{wb} have proved that for $2\leq{k}\leq{n}$, there exists a
$\bar{c}(k,n)$ such that there is no classical radial solution of
\eqref{dirichlet} and \eqref{asymptotic} if $c<\bar{c}(k,n)$.
\end{remark}

Recall that any real symmetric matrix $A$ has an eigen-decomposition
$A=O^{T}\Lambda{O}$ where $O$ is an orthogonal matrix, and $\Lambda$
is a diagonal matrix. That is, $A$ may be regarded as a real
diagonal matrix $\Lambda$ that has been re-expressed in some new
coordinate system, and the eigenvalues
$\lambda(A)=\lambda(\Lambda)$. Let
$$y=Ox,\quad\mbox{and}\quad\,v(y)=u(O^{-1}y),$$
then \eqref{dirichlet} and \eqref{asymptotic} become
$$
\begin{cases}
\sigma_{k}(\lambda(D^{2}_{y}v))=1,&\mbox{in}~\mathbb{R}^{n}\setminus{\overline{\widetilde{D}}},\\
v=\varphi(O^{-1}y),&\mbox{on}~\partial{\widetilde{D}}
\end{cases}
$$
and
$$\limsup_{|y|\rightarrow\infty}~\left(|O^{-1}y|^{\theta(n-2)}\left|~v(y)-\left(\frac{1}{2}y^{T}\Lambda{y}+bO^{-1}\cdot{y}+c\right)\right|~\right)<\infty,$$
where $\widetilde{D}$ is transformed from $D$ under $y=Ox$. So,
without loss of generality, we always assume that $A$ is diagonal in
this paper.

If $A$ is diagonal and $A\in\mathcal{A}_{n}$, then
$\sigma_{n}(\lambda(A))=1$, and we can find a diagonal matrix $Q$
with $\det{Q}=1$ such that $QAQ=I\in\mathcal{A}_{n}$. Clearly,
$\lambda(I)$ is not necessarily the same as $\lambda(A)$, but under
transformation $y=Qx$, we still have
$$\det\left(D^{2}_{x}u\right)=\det\left(QD^{2}_{y}uQ\right)= \det\left(D^{2}_{y}u\right).$$
Therefore, when the Monge-Amp\`{e}re equation is considered,
Caffarelli and Li \cite{cl} can assume without loss of generality
that $A=I$. However, when $2\leq{k}\leq{n}-1$, if $A$ is diagonal
and $A\in\mathcal{A}_{k}$, $\sigma_{k}(\lambda(A))=1$, although we
can also find a diagonal matrix $Q$ such that
$QAQ=c^{*}I\in\mathcal{A}_{k}$, it is clear that
$\lambda(A)\neq\lambda(c^{*}I)$ unless $A=c^{*}I$, and for Hessian
operator
$$\sigma_{k}\left(\lambda(QD^{2}_{y}uQ)\right)\neq\sigma_{k}(\lambda(Q))\sigma_{k}\left(\lambda(D^{2}_{y}u)\right)\sigma_{k}(\lambda(Q)).$$
So, in order to prove Theorem \ref{thm1}, we are only allowed to
assume that $A$ is diagonal, but we can not further assume that
$A=c^{*}I$.

\begin{definition}
For a diagonal matrix $A=\mathrm{diag}(a_{1},a_{2},\cdots,a_{n})$,
we call $u$ a generalized symmetric function with respect to $A$, if
$u$ is a function of
$$ s=\frac{1}{2}x^{T}Ax=\frac{1}{2}\sum_{i=1}^{n}a_{i}x_{i}^{2}. $$

If $u$ is a generalized symmetric function with respect to $A$ and
$u$ is a solution (respectively, subsolution, supersolution) of the
Hessian equation \eqref{sigmak}, then we call $u$ a generalized
symmetric solution (respectively, subsolution, supersolution) of
\eqref{sigmak}.
\end{definition}

In this paper we often abuse notations slightly by writing
$u(x)=u(\frac{1}{2}x^{T}Ax)$ for a generalized symmetric function
with respect to $A$. Clearly, for diagonal matrix
$A=\mathrm{diag}(a_{1},a_{2},\cdots,a_{n})\in\mathcal{A}_{k}$, and
real constants $\mu_{1}$, $\mu_{2}$, with $\mu_{1}^{k}=1$,
\begin{equation}\label{omegas_mu}
\omega(s)=\mu_{1}s+\mu_{2},\quad\,s=\frac{1}{2}\sum_{i=1}^{n}a_{i}x_{i}^{2}
\end{equation}
 satisfies the Hessian equation
\eqref{sigmak} and $\omega''(s)\equiv0$.

First, we will derive a formula of $\sigma_k(\lambda(M))$ for
matrices $M$ of the form
\begin{equation}\label{m1}
M=\Big(p_{i}\delta_{ij}-\beta\,q_{i}q_{j}\Big)_{n\times{n}},
\end{equation}
where $p=(p_{1},p_{2},\cdots,p_{n})$, $q=(q_{1},q_{2},\cdots,q_{n})$
and $\beta\in\mathbb{R}$.

\begin{prop}\label{propM}
If $M$ is a $n\times{n}$ matrix  of the form \eqref{m1} for
$p=(p_{1},p_{2},\cdots,p_{n})$, $q=(q_{1},q_{2},\cdots,q_{n})$ and
$\beta\in\mathbb{R}$, then we have
\begin{equation}\label{sigmakm}
\sigma_{k}(\lambda(M))=\sigma_{k}(p)-\beta\sum_{i=1}^{n}q_{i}^{2}\sigma_{k-1;i}(p),
\end{equation}
where $\sigma_{k-1;i}(p)=\sigma_{k-1}(p)|_{p_{i}=0}$.
\end{prop}

For any $A=\mathrm{diag}(a_{1},a_{2},\cdots,a_{n})$, suppose
$\omega\in{C}^{2}(\mathbb{R}^{n})$ is a generalized symmetric
function with respect to $A$, that is,
$$\omega(x)=\omega\left(\frac{1}{2}\sum_{i=1}^{n}a_{i}x_{i}^{2}\right),$$
then
$$D_{i}\omega(x)=\omega'(s)a_{i}x_{i},$$
$$ D_{ij}\omega(x)=\omega'(s)a_{i}\delta_{ij}+\omega''(s)(a_{i}x_{i})(a_{j}x_{j}).
$$
We have the following lemma.

\begin{lemma}\label{lem_sigmak}
For any $A=\mathrm{diag}(a_{1},a_{2},\cdots,a_{n})$, if
$\omega\in{C}^{2}(\mathbb{R}^{n})$ is a generalized symmetric
function with respect to $A$, then, with
$a=(a_{1},a_{2},\cdots,a_{n})$,
\begin{align}\label{sigmakomega}
\sigma_{k}(\lambda(D^{2}\omega))
=\sigma_{k}(a)(\omega')^{k}+\omega''(\omega')^{k-1}\sum_{i=1}^{n}\sigma_{k-1;i}(a)(a_{i}x_{i})^{2}.
\end{align}
\end{lemma}

If $A=c^{*}I$, $2\leq{k}\leq{n}$, then there exist a family of
radially symmetric functions
\begin{equation*}
\overline{\omega}_{k}(s)=\int_{1}^{s}\Big(1+\alpha\,t^{-\frac{n}{2}}\Big)^{\frac{1}{k}}dt,\quad\alpha>0,~s>0,
\end{equation*}
satisfying
$$\sigma_{k}(\lambda(D^{2}\omega))=1, \ \ \ \mbox{in}~\mathbb{R}^{n}\setminus\{0\}.$$
Such radially symmetric solutions play an important role to the
solvability of the exterior Dirichlet problems studied by
Caffarelli-Li \cite{cl} and by Dai-Bao \cite{db}. However, for any
given $A\in\mathcal{A}_{k}$ with $2\leq{k}\leq{n}-1$, it is not
enough to prove Theorem \ref{thm1} only using these radially
symmetric functions. Due to the invariance of \eqref{sigmak} for
$k=n$, the Monge-Amp\`{e}re equation, under affine transformations,
$\overline{\omega}_{n}(\frac{1}{2}x^{T}Ax)$ is a solution of
\eqref{sigmak} in $\mathbb{R}^{n}\setminus\{0\}$ for
$A\in\mathcal{A}_{n}$. So the Monge-Amp\`{e}re equation has
generalized symmetric solutions with respect to $A$ for every
$A\in\mathcal{A}_{n}$. A natural question is that whether
\eqref{sigmak} with $2\leq{k}\leq{n}-1$ has generalized symmetric
solutions with respect to $A$ for every $A\in\mathcal{A}_{k}$
besides those of the form \eqref{omegas_mu}.

For this, we have

\begin{prop}\label{prop1}
For $A=\mathrm{diag}(a_{1},a_{2},\cdots,a_{n})\in\mathcal{A}_{k}$,
$1\leq{k}\leq{n}$, and $0<\alpha<\beta<\infty$, if there exists an
$\omega\in{C}^{2}(\alpha,\beta)$ with $\omega''\nequiv0$ in
$(\alpha,\beta)$, such that
$\omega(x)=\omega(\frac{1}{2}\sum_{i=1}^{n}a_{i}x_{i}^{2})$ is
a generalized symmetric solution of the Hessian equation
\eqref{sigmak} in
$\{x\in\mathbb{R}^{n}~|~\alpha<\frac{1}{2}\sum_{i=1}^{n}a_{i}x_{i}^{2}<\beta\}$,
then
$$k=n\quad\mbox{or}\quad\,a_{1}=a_{2}=\cdots=a_{n}=c^{*},$$
where $c^{*}=(C_{n}^{k})^{-1/k}$, $C_{n}^{k}=\frac{n!}{(n-k)!k!}$,
and vice versa.
\end{prop}

This means that for
$A=\mathrm{diag}(a_{1},a_{2},\cdots,a_{n})\in\mathcal{A}_{k}$,
$2\leq{k}\leq{n}-1$, $\omega(\frac{1}{2}x^{T}Ax)$ is in general not
a solution of \eqref{sigmak}.

To prove Theorem \ref{thm1} for $2\leq{k}\leq{n}-1$, it suffices to
obtain enough subsolutions with appropriate properties. We construct
such subsolutions which are generalized symmetric functions with
respect to $A$. This is the main new ingredient in our proof of the
theorem.

This paper is set out as follows. In the next section we construct a
family of generalized symmetric smooth $k$-convex subsolutions of
\eqref{sigmak} in $\mathbb{R}^{n}\setminus\{0\}$. In Section
\ref{sec1}, we prove Theorem \ref{thm1} using Perron's method.

\section{Generalized symmetric solutions and subsolutions}\label{sec0}

In this section, we first derive formula \eqref{sigmakm} and
\eqref{sigmakomega}, then prove Proposition \ref{prop1}, and finally
construct a family of generalized symmetric smooth $k$-convex
subsolutions of \eqref{sigmak}.

For $A=\mathrm{diag}(a_{1},a_{2}\cdots,a_{n})$, we denote
$\lambda(A)=(a_{1},a_{2}\cdots,a_{n}):=a$. If $A\in\mathcal{A}_{k}$,
then we have $a_{i}>0$ $(i=1,2,\cdots,n)$ and $\sigma_{k}(a)=1$.
Here we introduce some notations. For any fixed $t$-tuple
$\{i_{1},\cdots,i_{t}\}$, $1\leq{t}\leq{n}-k$, we define
$$\sigma_{k;i_{1}\cdots{i}_{t}}(a)=\sigma_{k}(a)|_{a_{i_{1}}=\cdots=a_{i_{t}}=0},$$
that is, $\sigma_{k;i_{1}\cdots{i}_{t}}$ is the $k$-th order
elementary symmetric function of the $n-t$ variables
$\left\{a_{i}~\big|~i\in\{1,2,\cdots,n\}\setminus\{i_{1},i_{2},\cdots,i_{t}\}\right\}$.
The following properties of the functions $\sigma_{k}$ will be used
in this paper:
\begin{equation}\label{property}
\sigma_{k}(a)=\sigma_{k;i}(a)+a_{i}\sigma_{k-1;i}(a),\quad~i=1,2,\cdots,n,
\end{equation}
and
\begin{equation}\label{property1}
\sum_{i=1}^{n}a_{i}\sigma_{k-1;i}(a)=k\sigma_{k}(a).
\end{equation}

Now we prove Proposition \ref{propM} to derive a formula of
$\sigma_k(\lambda(M))$ for matrices $M$ of the form \eqref{m1}.

\begin{proof}[Proof of Proposition \ref{propM}]
If $\beta=0$, \eqref{sigmakm} is obvious. If $\beta\neq0$, we work
with
$$\widehat{M}=\frac{1}{\beta}M=(\hat{p}_{i}\delta_{ij}-q_{i}q_{j}),\quad\quad\hat{p}=\frac{p}{\beta}.$$
Therefore we only need to prove Proposition \ref{propM} for
$\beta=1$, which we assume in the rest of the proof.

Denote
\begin{equation}\label{det}
D_{n}\left(\{p_{1},p_{2},\cdots,p_{n}\};\{q_{1},q_{2},\cdots,q_{n}\};\lambda\right):=\det(\lambda{I}-M).
\end{equation}
By direct computations, we have
\begin{align*}
&D_{n}\left(\{p_{1},p_{2},\cdots,p_{n}\};\{q_{1},q_{2},\cdots,q_{n}\};\lambda\right)\\
&=\left|\begin{array}{ccccc}
                 \lambda-p_{1}+q_{1}^{2} & q_{1}q_{2} & \cdots & q_{1}q_{n-1} & q_{1}q_{n} \\
                 q_{2}q_{1} & \lambda-p_{2}+q_{2}^{2}  & \cdots & q_{2}q_{n-1} & q_{2}q_{n} \\
                 \cdots & \cdots & \cdots & \cdots & \cdots \\
                 q_{n-1}q_{1} & q_{n-1}q_{2} & \cdots & \lambda-p_{n-1}+q_{n-1}^{2} & q_{n-1}q_{n}\\
                 q_{n}q_{1} & q_{n}q_{2} & \cdots & q_{n}q_{n-1} & \lambda-p_{n}+q_{n}^{2}
               \end{array}\right|\\
&=\left|\begin{array}{ccccc}
                 \lambda-p_{1}+q_{1}^{2} & q_{1}q_{2} & \cdots & q_{1}q_{n-1} & 0 \\
                 q_{2}q_{1} & \lambda-p_{2}+q_{2}^{2}  & \cdots & q_{2}q_{n-1} & 0 \\
                 \cdots & \cdots & \cdots & \cdots & \cdots \\
                 q_{n-1}q_{1} & q_{n-1}q_{2} & \cdots & \lambda-p_{n-1}+q_{n-1}^{2} & 0\\
                 q_{n}q_{1} & q_{n}q_{2} & \cdots & q_{n}q_{n-1} & \lambda-p_{n}
               \end{array}\right|\\
&\hspace{1cm}+\left|\begin{array}{ccccc}
                 \lambda-p_{1}+q_{1}^{2} & q_{1}q_{2} & \cdots & q_{1}q_{n-1} & q_{1}q_{n} \\
                 q_{2}q_{1} & \lambda-p_{2}+q_{2}^{2}  & \cdots & q_{2}q_{n-1} & q_{2}q_{n} \\
                 \cdots & \cdots & \cdots & \cdots & \cdots \\
                 q_{n-1}q_{1} & q_{n-1}q_{2} & \cdots & \lambda-p_{n-1}+q_{n-1}^{2} & q_{n-1}q_{n}\\
                 q_{n}q_{1} & q_{n}q_{2} & \cdots & q_{n}q_{n-1} & q_{n}^{2}
               \end{array}\right|\\
&=(\lambda-p_{n})D_{n-1}\left(\{p_{1},p_{2},\cdots,p_{n-1}\};\{q_{1},q_{2},\cdots,q_{n-1}\};\lambda\right)\\
&\hspace{1cm}+q_{n}\left|\begin{array}{ccccc}
                 \lambda-p_{1}+q_{1}^{2} & q_{1}q_{2} & \cdots & q_{1}q_{n-1} & q_{1}q_{n} \\
                 q_{2}q_{1} & \lambda-p_{2}+q_{2}^{2}  & \cdots & q_{2}q_{n-1} & q_{2}q_{n} \\
                 \cdots & \cdots & \cdots & \cdots & \cdots \\
                 q_{n-1}q_{1} & q_{n-1}q_{2} & \cdots & \lambda-p_{n-1}+q_{n-1}^{2} & q_{n-1}q_{n}\\
                 q_{1} & q_{2} & \cdots & q_{n-1} & q_{n}
               \end{array}\right|.
\end{align*}
For the second term, multiplying its last row by $-q_{i}$
$(i\neq{n})$ and adding to the $i_{th}$ row, respectively, we obtain
\begin{align*}
&\left|\begin{array}{ccccc}
                 \lambda-p_{1}+q_{1}^{2} & q_{1}q_{2} & \cdots & q_{1}q_{n-1} & q_{1}q_{n} \\
                 q_{2}q_{1} & \lambda-p_{2}+q_{2}^{2}  & \cdots & q_{2}q_{n-1} & q_{2}q_{n} \\
                 \cdots & \cdots & \cdots & \cdots & \cdots \\
                 q_{n-1}q_{1} & q_{n-1}q_{2} & \cdots & \lambda-p_{n-1}+q_{n-1}^{2} & q_{n-1}q_{n}\\
                 q_{1} & q_{2} & \cdots & q_{n-1} & q_{n}
               \end{array}\right|\\
&=\left|\begin{array}{ccccc}
                 \lambda-p_{1} & 0 & \cdots & 0 & 0 \\
                 0 & \lambda-p_{2}  & \cdots & 0 & 0 \\
                 \cdots & \cdots & \cdots & \cdots & \cdots \\
                 0 & 0 & \cdots & \lambda-p_{n-1} & 0\\
                 q_{1} & q_{2} & \cdots & q_{n-1} & q_{n}
               \end{array}\right|\\
&=q_{n}(\lambda-p_{1})(\lambda-p_{2})\cdots(\lambda-p_{n-1}).
\end{align*}
Hence
\begin{align}\label{dn}
&D_{n}\left(\{p_{1},p_{2},\cdots,p_{n}\};\{q_{1},q_{2},\cdots,q_{n}\};\lambda\right)\nonumber\\
&=(\lambda-p_{n})D_{n-1}\left(\{p_{1},p_{2},\cdots,p_{n-1}\};\{q_{1},q_{2},\cdots,q_{n-1}\};\lambda\right)\nonumber\\
&\hspace{.5cm}+q_{n}^{2}(\lambda-p_{1})(\lambda-p_{2})\cdots(\lambda-p_{n-1}).
\end{align}
We will deduce from \eqref{dn}, by induction, that for $n\geq2$,
\begin{align}\label{dninduction}
&D_{n}\left(\{p_{1},p_{2},\cdots,p_{n}\};\{q_{1},q_{2},\cdots,q_{n}\};\lambda\right)
=\prod_{i=1}^{n}(\lambda-p_{i})+\sum_{j=1}^{n}\left(q_{j}^{2}\prod_{i\neq{j}}(\lambda-p_{i})\right).
\end{align}
For $n=2$,
\begin{align*}
D_{2}\left(\{p_{1},p_{2}\};\{q_{1},q_{2}\};\lambda\right) &=\left|
  \begin{array}{cc}
    \lambda-p_{1}+q_{1}^{2} & q_{1}q_{2} \\
    q_{1}q_{2} & \lambda-p_{2}+q_{2}^{2} \\
  \end{array}
\right|\\
&=(\lambda-p_{1})(\lambda-p_{2})+q_{1}^{2}(\lambda-p_{2})+q_{2}^{2}(\lambda-p_{1}).
\end{align*}
That is, \eqref{dninduction} holds for $n=2$. We now assume
\eqref{dninduction} holds for $n-1\geq2$. Then  by \eqref{dn} and
the induction hypothesis,
\begin{align*}
&D_{n}\left(\{p_{1},p_{2},\cdots,p_{n}\};\{q_{1},q_{2},\cdots,q_{n}\};\lambda\right)\nonumber\\
&=(\lambda-p_{n})D_{n-1}\left(\{p_{1},p_{2},\cdots,p_{n-1}\};\{q_{1},q_{2},\cdots,q_{n-1}\};\lambda\right)\\
& \hspace{.5cm} +q_{n}^{2}(\lambda-p_{1})(\lambda-p_{2})\cdots(\lambda-p_{n-1})\\
&=(\lambda-p_{n})\left(\prod_{i=1}^{n-1}(\lambda-p_{i})+\sum_{j=1}^{n-1}\left(q_{j}^{2}\prod_{i\neq{j},i\leq{n-1}}(\lambda-p_{i})\right)\right)\\
&\hspace{.5cm}+q_{n}^{2}(\lambda-p_{1})(\lambda-p_{2})\cdots(\lambda-p_{n-1})\\
&=\prod_{i=1}^{n}(\lambda-p_{i})+\sum_{j=1}^{n}\left(q_{j}^{2}\prod_{i\neq{j}}(\lambda-p_{i})\right).
\end{align*}
We have proved that \eqref{dninduction} holds for $n\geq2$. Recall
the Veite theorem that for any $n\times{n}$ matrix $U$,
\begin{equation}\label{veite1}
\det(\lambda{I}-U)=\sum_{i=0}^n(-1)^{i}\sigma_{i}(\lambda(U))\lambda^{n-i}.
\end{equation}
In particular, if $U=\mathrm{diag}(p_{1},p_{2}\cdots,p_{2})$,
\begin{equation}\label{veite2}
\prod_{i=1}^n(\lambda-p_i)=\sum_{i=0}^n(-1)^{i}\sigma_{i}(p)\lambda^{n-i},
\end{equation}
here $p=(p_{1},p_{2}\cdots,p_{n})$.  Using \eqref{det} and
\eqref{veite2}, \eqref{dninduction} is written as
\begin{align*}
\det(\lambda{I}-M)
&=\sum_{i=0}^{n}(-1)^{i}\sigma_{i}(p)\lambda^{n-i}+\sum_{j=1}^{n}\left(q_{j}^{2}\sum_{i=1}^{n}(-1)^{i-1}\sigma_{i-1;j}(p)\lambda^{n-i}\right)\nonumber\\
&=\sum_{i=0}^{n}(-1)^{i}\left(\sigma_{i}(p)-\sum_{j=1}^{n}q_{j}^{2}\sigma_{i-1;j}(p)\right)\lambda^{n-i},
\end{align*}
here we used standard conventions that $\sigma_{0}(p)=1$ and
$\sigma_{-1}(p)=0$. Thus, \eqref{sigmakm} follows from
\eqref{veite1}. The proof of Proposition \ref{propM} is completed.
\end{proof}

\begin{proof}[Proof of Lemma \ref{lem_sigmak}]
For any $A=\mathrm{diag}(a_{1},a_{2},\cdots,a_{n})$, if
$\omega\in{C}^{2}(\mathbb{R}^{n})$ is a generalized symmetric
function with respect to $A$, that is
$$\omega(x)=\omega\left(\frac{1}{2}\sum_{i=1}^{n}a_{i}x_{i}^{2}\right),$$ then
$$D_{i}\omega(x)=\omega'(s)a_{i}x_{i},$$
\begin{equation}\label{d2omega}
D_{ij}\omega(x)=\omega'(s)a_{i}\delta_{ij}+\omega''(s)(a_{i}x_{i})(a_{j}x_{j}).
\end{equation}
Comparing \eqref{m1} and \eqref{d2omega}, letting
$\beta=-\omega''(s)$, $p_{i}=\omega'(s)a_{i}$ and
$q_{i}=a_{i}x_{i}$, and substituting them into \eqref{sigmakm}, we
have \eqref{sigmakomega}.
\end{proof}

{\bf{Symmetric solutions.}} For $A=c^{*}I$ and $2\leq{k}\leq{n}$,
\begin{equation}\label{omegas=}
\overline{\omega}_{k}(s)=\int_{1}^{s}\Big(1+\alpha\,t^{-\frac{n}{2}}\Big)^{\frac{1}{k}}dt,\quad\alpha>0,~s>0,
\end{equation}
satisfies the ordinary differential equation
\begin{equation}\label{omegak=}
\sigma_{k}(\lambda(D^{2}\omega))=(\omega'(s))^{k}+2s\frac{k}{n}\omega''(s)(\omega'(s))^{k-1}=1,\quad\,s>0.
\end{equation}
Therefore,
$\overline{\omega}_{k}\left(\frac{c^{*}}{2}|x|^{2}\right)$ is a
solution of \eqref{sigmak} in $\mathbb{R}^{n}\setminus\{0\}$. In
order to prove Proposition \ref{prop1}, for every
$a=(a_{1},a_{2},\cdots,a_{n})\in\Gamma^{+}$, we denote
\begin{equation}\label{aki}
A_{k}^{i}(a)=a_{i}\sigma_{k-1;i}(a),\quad~i=1,2,\cdots,n.
\end{equation}
From the property of $\sigma_{k}$, \eqref{property1}, we have
\begin{equation}\label{property2}
\sum_{i=1}^{n}A_{k}^{i}(a)=k\sigma_{k}(a).
\end{equation}

\begin{proof}[Proof of Proposition \ref{prop1}.]
To better illustrate the idea of the proof , we start with $k=1$.
For $s\in(\alpha,\beta)$, $1\leq{i}\leq{n}$, let
$x=(0,\cdots,0,\sqrt{\frac{2s}{a_{i}}},0,\cdots,0)$. We have, using
$A\in\mathcal{A}_{1}$,
$$1=\Delta\omega(x)=\omega'(s)\sum_{j=1}^{n}a_{j}+\omega''(s)\sum_{j=1}^{n}a_{j}^{2}x_{j}^{2}=\omega'(s)+2s\omega''(s)a_{i}.$$
Since $\omega''\nequiv0$ in $(\alpha,\beta)$, there exists some
$\bar{s}\in(\alpha,\beta)$ such that $\omega''(\bar{s})\neq0$. It
follows that
$$a_{i}=\frac{1-\omega'(\bar{s})}{2\bar{s}\omega''(\bar{s})}$$
is independent of $i$. Since $A\in\mathcal{A}_{1}$,
$1=\sum_{i=1}^{n}a_{i}$. So $a_{1}=a_{2}=\cdots=a_{n}=\frac{1}{n}$.
Proposition \ref{prop1} for $k=1$ is established.

Now we consider the case $2\leq{k}\leq{n}$. For
$s\in(\alpha,\beta)$, $1\leq{i}\leq{n}$, let
$x=(0,\cdots,0,\sqrt{\frac{2s}{a_{i}}},0,\cdots,0)$, we have, using
Lemma \ref{lem_sigmak},
\begin{align*}
1&=\sigma_{k}(\lambda(D^{2}\omega(x)))\\
&=\sigma_{k}(a)(\omega'(s))^{k}+\omega''(s)(\omega'(s))^{k-1}\sigma_{k-1;j}(a)(a_{j}x_{j})^{2}\\
&=(\omega'(s))^{k}+2s\omega''(s)(\omega'(s))^{k-1}\sigma_{k-1;i}(a)a_{i}.
\end{align*}
It is clear from the above that $\omega'(s)\neq0$,
$\forall~s\in(\alpha,\beta)$. Since $\omega''\nequiv0$ in
$(\alpha,\beta)$, there exists some $\bar{s}\in(\alpha,\beta)$ such
that $\omega''(\bar{s})\neq0$. It follows that
$$A_{k}^{i}(a)=\sigma_{k-1;i}(a)a_{i}=\frac{1-(\omega'(\bar{s}))^{k}}{2\bar{s}\omega''(\bar{s})(\omega'(\bar{s}))^{k-1}}$$
is independent of $i$. For $2\leq{k}\leq{n}-1$, for any
$i_{1},i_{2}\in\{1,2,\cdots,n\}$, by \eqref{aki} and
\eqref{property}, we have
\begin{equation}\label{aki=}
\begin{aligned}
0&=A_{k}^{i_{1}}(a)-A_{k}^{i_{2}}(a)\\
&=a_{i_{1}}\sigma_{k-1;i_{1}}(a)-a_{i_{2}}\sigma_{k-1;i_{2}}(a)\\
&=a_{i_{1}}\left(a_{i_{2}}\sigma_{k-2;i_{1}i_{2}}(a)+\sigma_{k-1;i_{1}i_{2}}(a)\right)
-a_{i_{2}}\left(a_{i_{1}}\sigma_{k-2;i_{1}i_{2}}(a)+\sigma_{k-1;i_{1}i_{2}}(a)\right)\\
&=(a_{i_{1}}-a_{i_{2}})\sigma_{k-1;i_{1}i_{2}}(a).
\end{aligned}
\end{equation}
Since $a_{i}>0$, $i=1,2,\cdots,n$, it follows that
$\sigma_{k-1;i_{1}i_{2}}(a)\neq0$. By the arbitrariness of
$i_{1},i_{2}$, we have $a_{1}=a_{2}=\cdots=a_{n}$. Using
$\sigma_{k}(a)=1$, we have
$$a_{1}=a_{2}=\cdots=a_{n}=(C_{n}^{k})^{-1/k}.$$
Proposition \ref{prop1} is proved.
\end{proof}

{\bf{Generalized symmetric subsolutions.}} From Proposition
\ref{prop1}, we see that there is no generalized symmetric solutions
of \eqref{sigmak} with $\omega''(s)\nequiv0$ in remaining cases. We
will construct a family of generalized symmetric smooth functions
satisfying
$$\omega'(s)>0,\quad\omega''(s)\leq0,$$
and
$$\sigma_{k}(\lambda(D^{2}\omega))\geq1,\quad\mbox{and}\quad\sigma_{m}(\lambda(D^{2}\omega))\geq0,~1\leq{m}\leq{k}-1.$$
For $A=\mathrm{diag}(a_{1},a_{2},\cdots,a_{n})\in\mathcal{A}_{k}$,
denote $a=(a_{1},a_{2},\cdots,a_{n})$, and consider
\begin{align}\label{hk}
h_{k}(a):=\max\limits_{1\leq{i}\leq{n}}A_{k}^{i}(a).
\end{align}
Since $A_{n}^{i}(a)=a_{i}\sigma_{n-1;i}(a)=\sigma_{n}(a)$ for every
$i$, we have $h_{n}(a)=1$. By \eqref{aki}, \eqref{property} and
\eqref{property2}, we have, for $1\leq{k}\leq{n}-1$,
$$A_{k}^{i}(a)=a_{i}\sigma_{k-1;i}(a)<\sigma_{k}(a)=1,\quad\forall~i,$$
and
$$nh_{k}(a)\geq\sum_{i=1}^{n}A_{k}^{i}(a)=k\sigma_{k}(a)=k.$$
We see from the above that
\begin{equation}\label{hkinequality}
\frac{k}{n}\leq\,h_{k}(a)<1,
\end{equation}
with $``="$ holds if and only if $A_{k}^{i}(a)$ is independent of
$i$, i.e., in view of \eqref{aki=},
$a_{1}=a_{2}=\cdots=a_{n}=c^{*}$. For $n\geq3$ and
$2\leq{k}\leq{n}$,  in view of \eqref{hkinequality} and
$h_{n}(a)=1$, we have
\begin{equation}\label{k/hk>1}
\frac{k}{2h_{k}(a)}>1.
\end{equation}

By a simple computation, the following ordinary differential
equation
\begin{equation}\label{omega}
\begin{cases}
(\omega'(s))^{k}+2h_{k}(a)s\omega''(s)(\omega'(s))^{k-1}=1,\quad\,s>0,\\
\omega'(s)>0,\quad\omega''(s)\leq0
\end{cases}
\end{equation}
has a family of solutions
\begin{equation}\label{omegas}
\omega_{\alpha}(s)=\beta+\int_{\bar{s}}^{s}\left(1+\alpha\,t^{-\frac{k}{2h_{k}(a)}}\right)^{\frac{1}{k}}dt,
\quad\alpha>0,~s>0,
\end{equation}
where $\beta\in\mathbb{R}$ and $\bar{s}>0$. It follows from
\eqref{k/hk>1} that
\begin{align}\label{omega_alpha}
\omega_{\alpha}(s)
&=\beta+s-\bar{s}+\int_{\bar{s}}^{s}\left(\left(1+\alpha\,t^{-\frac{k}{2h_{k}(a)}}\right)^{\frac{1}{k}}-1\right)dt\nonumber\\
&=s+\mu(\alpha)+O\left(s^{\frac{(2-n)\theta}{2}}\right),\quad\mbox{as}\quad\,s\rightarrow\infty,
\end{align}
where
$$\mu(\alpha)=\beta-\bar{s}+\int_{\bar{s}}^{\infty}
\left(\left(1+\alpha\,t^{-\frac{k}{2h_{k}(a)}}\right)^{\frac{1}{k}}-1\right)dt<\infty,
$$ and
$$ \theta=\frac{1}{n-2}\left(\frac{k}{h_{k}(a)}-2\right).$$
We see from \eqref{hkinequality} that
$\theta\in\left(\frac{k-2}{n-2},1\right]$ if $2\leq{k}\leq{n}-1$,
and $\theta=1$ if $k=n$.

\begin{prop}\label{prop2.1}
For $n\geq3$ and $2\leq{k}\leq{n}$, $A\in\mathcal{A}_{k}$, let
$\omega_{\alpha}(x)=\omega_{\alpha}\left(\frac{1}{2}x^{T}Ax\right)$
be given in \eqref{omegas}. Then $\omega_{\alpha}$ is a smooth
$k$-convex subsolution of \eqref{sigmak} in
$\mathbb{R}^{n}\setminus\{0\}$ satisfying
\begin{equation}\label{omega_alpha(x)}
\omega_{\alpha}(x)=\frac{1}{2}x^{T}Ax+\mu(\alpha)+O\left(|x|^{\theta(2-n)}\right),\quad\quad\mbox{as}\quad\,x\rightarrow\infty.
\end{equation}
\end{prop}

\begin{proof}
Obviously, \eqref{omega_alpha(x)} follows from \eqref{omega_alpha}.
By computation,
$$\omega_{\alpha}'(s)=\Big(1+\alpha\,s^{-\frac{k}{2h_{k}(a)}}\Big)^{\frac{1}{k}}>1,$$
\begin{align}\label{omega''}
\omega_{\alpha}''(s)&=-\frac{1}{2h_{k}(a)s}\cdot\frac{\alpha}{s^{\frac{k}{2h_{k}(a)}}+\alpha}\cdot\omega_{\alpha}'(s)<0.
\end{align}
It is clear from Lemma \ref{lem_sigmak}, \eqref{hk} and
\eqref{omega} that
$$\sigma_{k}(\lambda(D^{2}u))\geq\sigma_{k}(a)(\omega_{\alpha}')^{k}+h_{k}(a)\omega_{\alpha}''(\omega_{\alpha}')^{k-1}2s=1,
\quad\mbox{in}~\mathbb{R}^{n}\setminus\{0\}.$$ By Lemma
\ref{lem_sigmak}, \eqref{omega''} and \eqref{hk},
 we have, for any $1\leq{m}\leq{k}-1$,
\begin{align*}
\sigma_{m}(\lambda(D^{2}u))
&=\sigma_{m}(a)(\omega_{\alpha}')^{m}+\omega_{\alpha}''(\omega_{\alpha}')^{m-1}\sum_{i=1}^{n}\sigma_{m-1;i}(a)(a_{i}x_{i})^{2}\\
&=(\omega_{\alpha}')^{m}\left(\sigma_{m}(a)-\frac{1}{2sh_{k}(a)}\cdot\frac{\alpha}{s^{\frac{k}{2h_{k}(a)}}+\alpha}
\sum_{i=1}^{n}\sigma_{m-1;i}(a)(a_{i}x_{i})^{2}\right)\\
&\geq(\omega_{\alpha}')^{m}\left(\sigma_{m}(a)-\frac{1}{2s}\cdot\frac{\alpha}{s^{\frac{k}{2h_{k}(a)}}+\alpha}\sum_{i=1}^{n}
\frac{\sigma_{m-1;i}(a)(a_{i}x_{i})^{2}}{a_{i}\sigma_{k-1;i}(a)}\right).
\end{align*}
In order to show $\sigma_{m}(\lambda(D^{2}u))\geq0$, it suffices to
prove, for each $1\leq{i}\leq{n}$,
\begin{equation}\label{sigmakinequality}
\sigma_{m}(a)\sigma_{k-1;i}(a)\geq\sigma_{m-1;i}(a).
\end{equation}
Note that the Newtonian inequalities may be expressed as
$$\frac{\sigma_{k+1}(a)}{C_{n}^{k+1}}\cdot\frac{\sigma_{k-1}(a)}{C_{n}^{k-1}}
\leq\left(\frac{\sigma_{k}(a)}{C_{n}^{k}}\right)^{2},$$ for
$1\leq{k}\leq{n-1}$. Since
$$\frac{C_{n}^{k-1}C_{n}^{k+1}}{C_{n}^{k}C_{n}^{k}}=\frac{(n-k)k}{(n-k+1)(k+1)}<1,$$
it follows that
$$\frac{\sigma_{k+1}(a)}{\sigma_{k}(a)}\leq\frac{\sigma_{k}(a)}{\sigma_{k-1}(a)},$$
which shows that the Hessian quotient
$\frac{\sigma_{k+1}(a)}{\sigma_{k}(a)}$ is decreasing with respect
to $k$. So we have for any $m\leq{k}$, and each $1\leq{i}\leq{n}$,
$$\sigma_{m;i}(a)\sigma_{k-1;i}(a)\geq\sigma_{m-1;i}(a)\sigma_{k;i}(a),$$
Then by the property \eqref{property}, it follows that
\begin{align*}
\sigma_{m}(a)\sigma_{k-1;i}(a)&=\left(\sigma_{m;i}(a)+a_{i}\sigma_{m-1;i}(a)\right)\sigma_{k-1;i}(a)\\
&\geq\sigma_{m-1;i}(a)\cdot\sigma_{k;i}(a)+\sigma_{m-1;i}(a)\cdot\,a_{i}\sigma_{k-1;i}(a)\\
&=\sigma_{m-1;i}(a)\sigma_{k}(a)\\
&=\sigma_{m-1;i}(a).
\end{align*}
i.e. \eqref{sigmakinequality} is proved. Hence $\omega_{\alpha}$ is
a smooth $k$-convex subsolution of \eqref{sigmak} in
$\mathbb{R}^{n}\setminus\{0\}$.
\end{proof}

\section{Proof of Theorem \ref{thm1}}\label{sec1}

The following Lemma holds for any invertible and symmetric matrix
$A$, and $A$ is not necessarily diagonal or in $\mathcal{A}_{k}$,
$2\leq{k}\leq{n}$.

\begin{lemma}\label{lem_omega}
Let $\varphi\in{C}^{2}(\partial{D})$. There exists some constant
$C$, depending only on $n$, $\|\varphi\|_{C^{2}(\partial{D})}$, the
upper bound of $A$, the diameter and the convexity of $D$,  and the
$C^{2}$ norm of $\partial{D}$, such that, for every
$\xi\in\partial{D}$, there exists $\bar{x}(\xi)\in\mathbb{R}^{n}$
satisfying
$$|\bar{x}(\xi)|\leq\,C\quad\mbox{and}\quad\,w_{\xi}<\varphi~\mbox{on}~\overline{D}\setminus\{\xi\},$$
where
$$w_{\xi}(x):=\varphi(\xi)+\frac{1}{2}\left((x-\bar{x}(\xi))^{T}A(x-\bar{x}(\xi))-(\xi-\bar{x}(\xi))^{T}A(\xi-\bar{x}(\xi))\right),\quad\,x\in\mathbb{R}^{n}.$$
\end{lemma}

\begin{proof}
Let $\xi\in\partial{D}$. By a translation and a rotation, we may
assume without loss of generality that $\xi=0$ and $\partial{D}$ is
locally represented by the graph of
$$x_{n}=\rho(x')=O(|x'|^{2}),$$
and $\varphi$ locally has the expansion
\begin{align*}
\varphi(x',\rho(x'))&=\varphi(0)+\varphi_{x_{1}}(0)x_{1}+\cdots+\varphi_{x_{n}}(0)x_{n}
+O(|x|^{2})\\
&=\varphi(0)+\varphi_{x_{1}}(0)x_{1}+\cdots+\varphi_{x_{n-1}}(0)x_{n-1}
+O(|x'|^{2}),
\end{align*}
where $x'=(x_{1},\cdots,x_{n-1})$.

Since $A$ is invertible, we can find
$\bar{x}=\bar{x}(t)\in\mathbb{R}^{n}$ such that, for appropriate $t$
to fit our need later,
$$A\bar{x}(t)=\left(-\varphi_{x_{1}}(0),\cdots,-\varphi_{x_{n-1}}(0),t\right)^{T}.$$
Let
$$w(x)=\varphi(0)+\frac{1}{2}\left((x-\bar{x})^{T}A(x-\bar{x})-\bar{x}^{T}A\bar{x}\right),\quad\,x\in\mathbb{R}^{n}.$$
Then
\begin{equation}\label{w(x)}
w(x)=\varphi(0)+\frac{1}{2}x^{T}Ax-x^{T}A\bar{x}=\varphi(0)+\frac{1}{2}x^{T}Ax+\sum_{\alpha=1}^{n-1}\varphi_{x_{\alpha}}(0)x_{\alpha}-tx_{n}.
\end{equation}
It follows that
\begin{align*}
(w-\varphi)(x',\rho(x'))&=\frac{1}{2}x^{T}Ax
-t\rho(x')+O(|x'|^{2})\\
&\leq\,C\left(|x'|^{2}+\rho(x')^{2}\right) -t\rho(x'),
\end{align*}
where $C$ depends only on the upper bound of $A$,
$\|\varphi\|_{C^{2}(\partial{D})}$, and the $C^{2}$ norm of
$\partial{D}$. By the strict convexity of $\partial{D}$, there
exists some constant $\delta>0$ depending only on $D$ such that
\begin{equation}\label{convex}
\rho(x')\geq\delta|x'|^{2},\quad\forall~|x'|<\delta.
\end{equation}
Clearly, for large $t$, we have
$$(w-\varphi)(x',\rho(x'))<0,\quad\forall~0<|x'|<\delta.$$
The largeness of $t$ depends only on
$\delta,A,\|\varphi\|_{C^{2}(\partial{D})}$, and the $C^{2}$ norm of
$\partial{D}$.

On the other hand, by the strict convexity of $\partial{D}$ and
\eqref{convex},
$$x_{n}\geq\delta^{3},\quad\forall~x\in\partial{D}\setminus\{(x',\rho(x'))~\big|~|x'|<\delta\}.$$
It follows from \eqref{w(x)} that
$$w(x)\leq\,C-\delta^{3}t,\quad~\forall~
x\in\partial{D}\setminus\{(x',\rho(x'))~\big|~|x'|<\delta\},$$ where
$C$ depends only on
$A,\mathrm{diam}(D),\|\varphi\|_{C^{2}(\partial{D})}$. By making $t$
large (still under control), we have
$$w(x)-\varphi(x)<0,\quad~\forall~
x\in\partial{D}\setminus\{(x',\rho(x'))~\big|~|x'|<\delta\}.$$ Lemma
\ref{lem_omega} is established.
\end{proof}

By an orthogonal transformation and by subtracting a linear function
from $u$, we only need to prove Theorem \ref{thm1} for the case that
$A=\mathrm{diag}(a_{1},a_{2},\cdots,a_{n})$ where $a_{i}>0$
$(1\leq{i}\leq{n})$, $b=0$.

\begin{proof}[Proof of Theorem \ref{thm1}]
Without loss of generality, we assume that $0\in D$. For $s>0$, let
$$ E(s):=\left\{x\in\mathbb{R}^n~\big|~\frac12 x^TAx<s \right\}. $$
Fix $\bar{s}>0$ such that $\overline{D}\subset E(\bar{s})$. For $\alpha>0$, $\beta\in\mathbb{R}$, set
$$ \omega_{\alpha}(x)=\beta+\int^{\frac{1}{2}x^{T}Ax}_{\bar{s}}
\left(1+\alpha{t}^{-\frac{k}{2h_{k}(a)}}\right)^{\frac{1}{k}}dt,$$
as in \eqref{omegas}. We have by Proposition \ref{prop2.1} that
$\omega_{\alpha}$ is a smooth $k$-convex subsolution of
\eqref{sigmak} in $\mathbb{R}^{n}\setminus\{0\}$, and
$$ \omega_{\alpha}(x)=\frac{1}{2}x^{T}Ax+\mu(\alpha)+O\left(|x|^{\theta(2-n)}\right),
\quad\quad\mbox{as}\quad\,x\rightarrow\infty. $$
Here
$$ \mu(\alpha)=\beta-\bar{s}+\int_{\bar{s}}^{\infty}
\left(\left(1+\alpha\,t^{-\frac{k}{2h_{k}(a)}}\right)^{\frac{1}{k}}-1\right)dt,
\quad\quad\theta\in\left[\frac{k-2}{n-2},1\right] . $$ Clearly,
$\mu(\alpha)$ is strictly increasing in $\alpha$, and
\begin{equation}\label{h2-1}
\lim_{\alpha\rightarrow\infty}\mu(\alpha)=\infty.
\end{equation}
On the other hand,
\begin{equation}\label{h1-1}
\omega_{\alpha}\leq\beta,\quad\quad\mbox{in}~E(\bar{s})\setminus\overline{D},~\forall~\alpha>0.
\end{equation}

Let
$$ \beta:=\min\left\{w_{\xi}(x)~\big|~\xi\in\partial{D},~x\in\overline{E(\bar{s})}\setminus{D}\right\},$$
$$ \widehat{b}:=\max\left\{w_{\xi}(x)~\big|~\xi\in\partial{D},~x\in\overline{E(\bar{s})}\setminus{D}\right\},$$
where $w_\xi(x)$ is given by Lemma \ref{lem_omega}. We will fix the
value of $c_{*}$ in the proof. First we require that $c_{*}$
satisfies $c_{*}>\widehat{b}$. It follows that
$$ \mu(0)=\beta-\bar{s}<\beta\leq\widehat{b}<c_{*}. $$
Thus, in view of \eqref{h2-1}, for every $c>c_{*}$, there exists a unique $\alpha(c)$ such that
\begin{equation}\label{h4-1}
\mu(\alpha(c))=c.
\end{equation}
So $\omega_{\alpha(c)}$ satisfies
\begin{equation}\label{h5-1}
\omega_{\alpha(c)}(x)=\frac{1}{2}x^{T}Ax+c+O\left(|x|^{\theta(2-n)}\right),
\quad\quad\mbox{as}\quad\,x\rightarrow\infty.
\end{equation}
Set
$$\underline{w}(x)=\max\left\{w_{\xi}(x)~\big|~\xi\in\partial{D}\right\}.$$
It is clear by Lemma \ref{lem_omega} that $\underline{w}$ is a
locally Lipschitz function in $\mathbb{R}^{n}\setminus{D}$, and
$\underline{w}=\varphi$ on $\partial{D}$. Since $w_{\xi}$ is a
smooth convex solution of \eqref{sigmak}, $\underline{w}$ is a
viscosity subsolution of \eqref{sigmak} in
$\mathbb{R}^{n}\setminus\overline{D}$. We fix a number
$\hat{s}>\bar{s}$, and then choose another number
$\widehat{\alpha}>0$ such that
$$ \min_{\partial E(\hat{s})}\omega_{\widehat{\alpha}} > \max_{\partial E(\hat{s})}\underline{w}. $$
We require that $c_{*}$ also satisfies $c_{*}\geq\mu(\widehat{\alpha})$. We now fix the value
of $c_{*}$.

For $c\geq{c}_{*}$, we have
$\alpha(c)={\mu}^{-1}(c)\geq{\mu}^{-1}(c_{*})\geq\widehat{\alpha}$,
and therefore
\begin{equation}\label{h7-1}
\omega_{\alpha(c)}\geq\omega_{\widehat{\alpha}}>\underline{w},\quad\mbox{on}~\partial{E(\hat{s})}.
\end{equation}
By \eqref{h1-1}, we have
\begin{equation}\label{h7-2}
\omega_{\alpha(c)}\leq \beta\leq\underline{w},\quad\mbox{in}~E(\bar{s})\setminus{\overline{D}}.
\end{equation}
Now we define, for $c>{c}_{*}$,
$$
\underline{u}(x)=
\begin{cases}
\max\left\{\omega_{\alpha(c)}(x),\underline{w}(x)\right\},&x\in E(\hat{s})\setminus{D},\\
\omega_{\alpha(c)}(x),&x\in\mathbb{R}^{n}\setminus E(\hat{s}).
\end{cases}
$$
We know from \eqref{h7-2} that
\begin{equation}\label{h8-1}
\underline{u}=\underline{w},\quad\mbox{in}~E(\bar{s})\setminus{\overline{D}},
\end{equation}
and in particular
\begin{equation}\label{h8-1}
\underline{u}=\underline{w}=\varphi,\quad\mbox{on}~\partial{D}.
\end{equation}
We know from \eqref{h7-1} that $\underline{u}=\omega_{\alpha(c)}$ in
a neighborhood of $\partial E(\hat{s})$. Therefore $\underline{u}$
is locally Lipschitz in $\mathbb{R}^{n}\setminus{D}$. Since both
$\omega_{\alpha(c)}$ and $\underline{w}$ are viscosity subsolutions
of \eqref{sigmak} in $\mathbb{R}^{n}\setminus\overline{D}$, so is
$\underline{u}$.

For $c>{c}_{*}$,
$$\overline{u}(x):=\frac{1}{2}x^{T}Ax+c$$
is a smooth convex solution of \eqref{sigmak}. By \eqref{h7-2},
$$ \omega_{\alpha(c)}\leq\beta\leq\widehat{b}<c_*<\overline{u},\quad\mbox{on}~\partial{D}.$$
We also know by \eqref{h5-1} that
$$\lim_{|x|\rightarrow\infty}\left(\omega_{\alpha(c)}(x)-\overline{u}(x)\right)=0.$$
Thus, in view of the comparison principle for smooth $k$-convex
solutions of \eqref{sigmak}, (see \cite{cns}), we have
\begin{equation}\label{h10-1}
\omega_{\alpha(c)}\leq\overline{u},\quad\mbox{on}~\mathbb{R}^{n}\setminus{D}.
\end{equation}
By \eqref{h7-1} and the above, we have, for $c>{c}_{*}$,
$$w_{\xi}\leq\overline{u},\quad\mbox{on}~\partial(E(\hat{s})\setminus{D}),
~\forall~\xi\in\partial{D}.$$
By the comparison principle for smooth convex solutions of \eqref{sigmak}, we have
$$w_{\xi}\leq\overline{u},\quad\mbox{in}~E(\hat{s})\setminus{\overline{D}},~\forall~\xi\in\partial{D}.$$
Thus
$$\underline{w}\leq\overline{u},\quad\mbox{in}~E(\hat{s})\setminus{\overline{D}}.$$
This, combining with \eqref{h10-1}, implies that
$$\underline{u}\leq\overline{u},\quad\mbox{in}~\mathbb{R}^{n}\setminus{D}.$$

For any $c>c_{*}$, let $\mathcal{S}_{c}$ denote the set of
$v\in\mathrm{USC}(\mathbb{R}^{n}\setminus{D})$ which are viscosity
subsolutions of \eqref{sigmak} in
$\mathbb{R}^{n}\setminus\overline{D}$ satisfying
\begin{equation}\label{h13-1}
v=\varphi,\quad\mbox{on}~\partial{D},
\end{equation}
and
\begin{equation}\label{h13-2}
\underline{u}\leq{v}\leq\overline{u},\quad\mbox{in}~\mathbb{R}^{n}\setminus{D}.
\end{equation}
We know that $\underline{u}\in\mathcal{S}_{c}$. Let
$$u(x):=\sup\left\{v(x)~|~v\in\mathcal{S}_{c}\right\}, \quad x \in \mathbb{R}^{n}\setminus{D}.$$
By \eqref{h5-1}, and the definitions of $\underline{u}$ and $\overline{u}$,
\begin{equation}\label{h13-3}
u(x)\geq\underline{u}(x)=\omega_{\alpha(c)}(x)=\frac{1}{2}x^{T}Ax+c+O\left(|x|^{\theta(2-n)}\right),
\quad\quad\mbox{as}\quad\,x\rightarrow\infty.
\end{equation}
and
$$ u(x)\leq \overline{u}(x)=\frac{1}{2}x^{T}Ax+c. $$
The estimate \eqref{asymptotic} follows.

Next, we prove that $u$ satisfies the boundary condition. It is
obvious from \eqref{h8-1} that
$$ \liminf_{x\rightarrow\xi}u(x)\geq \lim_{x\rightarrow\xi}\underline{u}(x)
=\varphi(\xi),\quad\forall~\xi\in\partial{D}.$$
So we only need to prove that
$$\limsup_{x\rightarrow\xi}u(x)\leq \varphi(\xi),\quad\forall~\xi\in\partial{D}.$$
Let $\omega_{c}^{+}\in C^2(\overline{E(\bar{s})\setminus D})$ be defined by
$$
\begin{cases}
\Delta\omega_{c}^{+}=0,&\mbox{in}~E(\bar{s})\setminus{\overline{D}},\\
\omega_{c}^{+}=\varphi,&\mbox{on}~\partial{D},\\
\omega_{c}^{+}=\max\limits_{\partial{E(\bar{s})}}\overline{u}
=\bar{s}+c,&\mbox{on}~\partial{E(\bar{s})}.
\end{cases}
$$
It is easy to see that a viscosity subsolution $v$ of \eqref{sigmak}
satisfies $\Delta{v}\geq0$ in viscosity sense. Therefore, for every
$v\in\mathcal{S}_{c}$, by $v\leq\omega_{c}^{+}$ on
$\partial(E(\bar{s})\setminus{D})$, we have
$$v\leq\omega_{c}^{+}\quad\mbox{in}~E(\bar{s})\setminus{\overline{D}}.$$
It follows that
$$u\leq\omega_{c}^{+}\quad\mbox{in}~E(\bar{s})\setminus{\overline{D}},$$
and then
$$\limsup_{x\rightarrow\xi}u(x)\leq \lim_{x\rightarrow\xi}\omega_{c}^{+}(x)
=\varphi(\xi),\quad\forall~\xi\in\partial{D}.$$

Finally, we prove that $u$ is a viscosity solution of
\eqref{sigmak}. The following ingredients for the viscosity
adaptation of Perron's method (see \cite{ishii}) are available.

\begin{lemma}
Let $\Omega\subset\mathbb{R}^{n}$ be a bounded open set,
$u\in\mathrm{LSC}(\overline{\Omega})$ and
$v\in\mathrm{USC}(\overline{\Omega})$ are respectively viscosity
supersolutions and subsolutions of \eqref{sigmak} in $\Omega$ satisfying
$u\geq{v}$ on $\partial{\Omega}$. Then $u\geq{v}$ in $\Omega$.
\end{lemma}

Under the assumptions $u,v\in{C}^{0}(\overline{\Omega})$, the lemma
was proved in \cite{u}, based on Jensen approximations (see
\cite{jen}). The proof remains valid under the weaker regularity
assumptions on $u$ and $v$.

\begin{lemma}\label{lemmaB}
Let $\Omega\subset\mathbb{R}^{n}$ be an open set, and let
$\mathcal{S}$ be a non-empty family of viscosity subsolutions
(supersolutions) of \eqref{sigmak} in $\Omega$. Set
$$u(x)=\sup~(\inf)~\left\{v(x)~|~v\in\mathcal{S}\right\},$$
and
$$u^{*}~(u_{*})~(x)=\lim_{r\rightarrow0}\sup_{B_{r}}~(\inf_{B_{r}})~u$$
be the upper (lower) semicontinuous envelope of $u$. Then, if
$u^{*}<\infty$ ($u_{*}>-\infty$) in $\Omega$, $u^{*}$ ($u_{*}$) is a
viscosity subsolution (supersolution) of \eqref{sigmak} in $\Omega$.
\end{lemma}
Lemma \ref{lemmaB} can be proved by standard arguments, see e.g.
\cite{cil}. With these ingredients, an application of the Perron
process, see e.g. Lemma 4.4 in \cite{cil}, gives that
$u\in{C}^{0}(\mathbb{R}^{n}\setminus{D})$ is a viscosity solution of
\eqref{dirichlet}. Theorem \ref{thm1} is established.
\end{proof}

\noindent{\bf{\large Acknowledgements.}} The first author was
partially supported by NNSF (11071020) and SRFDPHE (20100003110003).
He also would like to thank the Department of Mathematics and the
Center for Nonlinear Analysis at Rutgers University for the
hospitality and the stimulating environment. The second author was
partially supported by SRFDPHE (20100003120005), NNSF (11071020)
(11126038) and Ky and Yu-Fen Fan Fund Travel Grant from the AMS. The
work of the third author was partially supported by NSF grant
DMS-0701545. They were all partially supported by Program for
Changjiang Scholars and Innovative Research Team in University in
China.


\begin{thebibliography}{99}

\bibitem{c} L. A. Caffarelli:  Topics in PDEs: The Monge-Amp\`{e}re equation. Graduate course, Courant Institute, New York University, 1995.

\bibitem{cc} L. A. Caffarelli and X. Cabr\'{e}: Fully nonlinear
elliptic equations. American Mathematical Society Colloquium
Publications, 43 AMS, Providence, R.I., 1995.

\bibitem{cl} L. Caffarelli and Y. Y. Li: An extension to a theorem of J\"{o}rgens, Calabi, and Pogorelov. Comm. Pure Appl. Math. 56 (2003), 549-583.

\bibitem{cns} L. A. Caffarelli, L. Nirenberg and J. Spruck: The Dirichlet problem for nonlinear second-order elliptic equations,
III. Functions of the eigenvalues of the Hessian. Acta Math. 155
(1985), 261-301.

\bibitem{ce} E. Calabi: Improper affine hyperspheres of convex type and a generalization of a theorem by K. J\"{o}rgens. Michigan Math. J. 5 (1958), 105-126.

\bibitem{cy} S. Y. Cheng and S. T. Yau: Complete affine hypersurfaces, I. The completeness of affine metrics. Comm. Pure Appl. Math. 39 (1986), 839-866.

\bibitem{cw} K. S. Chou and X.-J. Wang: A variational theory of the Hessian equation. Comm. Pure Appl. Math. 54 (2001), 1029-1064.

\bibitem{cil} M. G. Crandall, H. Ishii and P.-L. Lions:  User's
guide to viscosity solutions of second order partial differential
equations. Bull. Amer. Math. Soc. 27 (1992), 1-67.

\bibitem{d1} L. M. Dai: Existence of solutions with asymptotic behavior of exterior problems of Hessian
equations. Proc. Amer. Math. Soc. 139 (2011), 2853-2861.

\bibitem{db} L. M. Dai and J. G. Bao: On uniqueness and extence of
viscosity solutions to Hessian equations in exterior domains. Front.
Math. China. 6 (2011), 221-230.


\bibitem{d} P. Delano\"{e}: Partial decay on simple manifolds. Ann. Global Anal.
Geom. 10 (1992), 3-61.

\bibitem{fmm1} L. Ferrer, A. Mart\'{i}nez and F. Mil\'{a}n: An extension of a theorem by K. J\"{o}rgens and a maximum
principle at infinity for parabolic affine spheres. Math. Z. 230
(1999), 471-486.

\bibitem{fmm2} L. Ferrer, A. Mart\'{i}nez and F. Mil\'{a}n: The space of
parabolic affine spheres with fixed compact boundary. Monatsh. Math.
130 (2000), 19-27.

\bibitem{ishii} H. Ishii: On uniqueness and existence of viscosity solutions of
fully nonlinear second-order elliptic PDEs. Comm. Pure Appl. Math.
42 (1989), 15-45.

\bibitem{jen} R. Jensen: The maximum principle for viscosity solutions of
fully nonlinear second order partial differential equations. Arch.
Rational Mech. Anal. 101 (1988),  1-27.

\bibitem{jian} H. Y. Jian: Hessian equations with infinite Dirichlet boundary value. Indiana
Univ. Math. J. 55 (2006), 1045-1062.

\bibitem{j} K. J\"{o}rgens: \"{U}ber die L\"{o}sungen der Differentialgleichung $rt-s^2=1$. Math. Ann. 127 (1954), 130-134.

\bibitem{jx} J. Jost and Y. L. Xin: Some aspects of the global geometry of entire space-like submanifolds.
Results Math. 40 (2001), 233-245.

\bibitem{ms} N. Meyers and J. Serrin: The exterior Dirichlet problem
for second order elliptic partial differential equations. J. Math.
Mech. 9 (1960), 513-538.

\bibitem{p} A. V. Pogorelov: On the improper convex affine hyperspheres. Geometriae Dedicata,  1 (1972), 33-46.

\bibitem{t0} N. S. Trudinger: The Dirichlet problem for the prescribed curvature
equations. Arch. Rational Mech. Anal. 111 (1990), 153-179.

\bibitem{t1} N. S. Trudinger: On the Dirichlet problem for Hessian equations. Acta Math. 175 (1995), 151-164.

\bibitem{t2} N. S. Trudinger: Weak solutions of Hessian equations. Comm. Part. Diff. Equ. 22 (1997), 1251-1261.

\bibitem{tw} N. S. Trudinger and X.-J. Wang: The Bernstein problem
for affine maximal hypersurface. Invent. Math. 140 (2000), 399-422.

\bibitem{u} J. I. E. Urbas: On the existence of nonclassical solutions for two class of fully nonlinear elliptic equations. Indiana
Univ. Math. J. 39 (1990), 355-382.

\bibitem{wb} C. Wang and J. G. Bao: Necessary and sufficient conditions
 on existence and convexity of solutions for Dirichlet problems of Hessian equations on exterior
 domains. preprint.





\end{thebibliography}
\end{document}